\documentclass[12pt,letterpaper]{article}
\usepackage[left=0.9in,top=0.9in,right=0.9in,bottom=0.9in]{geometry}

\usepackage{amsfonts,amsmath,amsthm,amssymb,latexsym,mathrsfs,stmaryrd}
\usepackage{hyperref}

\theoremstyle{theorem}
\newtheorem{theorem}{Theorem}[section]
\newtheorem{corollary}[theorem]{Corollary}
\newtheorem{lemma}[theorem]{Lemma}
\newtheorem{proposition}[theorem]{Proposition}

\theoremstyle{definition}

\newtheorem{conjecture}[theorem]{Conjecture}
\newtheorem{question}[theorem]{Question}

\theoremstyle{remark}
\newtheorem{remark}[theorem]{Remark}

\numberwithin{equation}{section}

\usepackage{amsfonts,amsmath,amsthm,amssymb,latexsym,mathrsfs,stmaryrd} \usepackage{mathtools}

\usepackage{tikz}
\usetikzlibrary{matrix,arrows,decorations.pathmorphing}
\usepackage{tikz-cd}
\tikzset{commutative diagrams/.cd}

\newcommand\C{{\mathbb C}}
\newcommand\A{{\mathbb A}}

\newcommand\Fq{{{\mathbb F}_q}}

\def\O{\mathcal O}

\DeclareMathOperator{\im}{im}

\DeclareMathOperator{\Hom}{Hom}

\DeclareMathOperator{\Aut}{Aut}

\DeclareMathOperator{\Spec}{Spec}

\DeclareMathOperator{\GL}{GL}

\DeclareMathOperator{\Mat}{Mat}

\newcommand\subeq{\subseteq}

\newcommand\supeq{\supseteq}

\newcommand\ttilde{\widetilde}
\newcommand\hhat{\widehat}
\newcommand\onto{\twoheadrightarrow}

\DeclarePairedDelimiter{\abs}{\lvert}{\rvert}

\DeclarePairedDelimiter{\set}{\{}{\}}

\DeclarePairedDelimiter{\parens}{\lparen}{\rparen}
\DeclarePairedDelimiter{\bracks}{\lbrack}{\rbrack}
\DeclarePairedDelimiter{\bbracks}{\llbracket}{\rrbracket}

\DeclarePairedDelimiter\ceil{\lceil}{\rceil}

\DeclarePairedDelimiter{\ev}{.}{\rvert}

\newcommand{\tu}{\textup}

\usepackage{enumitem}
\setenumerate[1]{label=\textup{(\alph*)}}
\setenumerate[2]{label=\textup{(\roman*)}}

\usepackage[main=english,romanian]{babel}

\newcommand{\KVar}{{K_0(\mathrm{Var}_\C)}}
\renewcommand{\L}{\mathbb{L}}

\DeclareMathOperator{\Nilp}{Nilp}
\DeclareMathOperator{\Hilb}{Hilb}
\DeclareMathOperator{\cl}{cl}

\DeclareMathOperator{\Durf}{Durf}
\newcommand{\Zhat}{\hhat Z}
\newcommand{\Ohat}{\hhat \O}
\newcommand{\zetahat}{\hhat \zeta}

\newcommand{\Coh}{\mathrm{Coh}}
\newcommand{\supp}{\mathrm {supp}}

\newcommand{\partialtheta}{\Theta_{\mathrm p}}
\newcommand{\wzero}{{\textup{ w/ } [0]}}

\usepackage{ytableau}

\begin{document}
\title{Mutually annihilating matrices, and a Cohen--Lenstra series for the nodal singularity}
\author{Yifeng Huang\footnote{Dept.\ of Mathematics, University of British Columbia. Email Address: {\tt huangyf@math.ubc.ca}. }}
\date{\today}

\bibliographystyle{abbrv}

\maketitle

\begin{abstract}
We give a generating function for the number of pairs of $n\times n$ matrices $(A, B)$ over a finite field that are mutually annihilating, namely, $AB=BA=0$. This generating function can be viewed as a singular analogue of a series considered by Cohen and Lenstra. We show that this generating function has a factorization that allows it to be meromorphically extended to the entire complex plane. We also use it to count pairs of mutually annihilating nilpotent matrices. This work is essentially a study of the motivic aspects about the variety of modules over $\C[u,v]/(uv)$ as well as the moduli stack of coherent sheaves over an algebraic curve with nodal singularities. 
\end{abstract}

\section{Introduction}

\subsection{History and Motivation}
Let $R$ be a commutative ring with only finite quotient fields. We define the \emph{Cohen--Lenstra zeta function} of $R$ as
\begin{equation}
\zetahat_R(s):=\sum_M \dfrac{1}{\abs{\Aut M}} \abs{M}^{-s},
\end{equation}
where $M$ ranges over all isomorphism classes of finite(-cardinality) $R$-modules. 

When $R$ is a Dedekind domain, this function has been considered in the important work of Cohen and Lenstra \cite{cohenlenstra1984heuristics} about the statistics of finite $R$-modules, motivated by the distribution of class groups of imaginary quadratic fields. They proved a simple formula for it in \cite[p.\ 39]{cohenlenstra1984heuristics}, which was crucial in their work. 

If $R$ contains a finite field $\Fq$ with $q$ elements, we define the \emph{Cohen--Lenstra series} of $R$ over $\Fq$ as
\begin{equation}
\Zhat_{R/\Fq}(x):=\sum_M \dfrac{1}{\abs{\Aut M}} x^{\dim_\Fq M},
\end{equation}
where $M$ ranges over all isomorphism classes of finite $R$-modules. Clearly $\zetahat_R(s)=\Zhat_{R/\Fq}(q^{-s})$. When the ground field is clear from the context, we may simply denote $\Zhat_{R/\Fq}(x)$ by $\Zhat_R(x)$. 

Existing work in various areas of mathematics can be put in the context of the Cohen--Lenstra series:
\begin{itemize}
\item When $R=\Fq$, the series $\Zhat_{\Fq}(x)$ is the subject of Rogers--Ramanujan identities \cite[p.\ 104]{andrewspartitions}, which state that $\Zhat_{\Fq}(1)$ and $\Zhat_{\Fq}(q^{-1})$ each equals to an infinite product:
\begin{gather}
\Zhat_{\Fq}(1)=\frac{1}{(q^{-1};q^{-5})_\infty (q^{-4};q^{-5})_\infty};\\
\Zhat_{\Fq}(q^{-1})=\frac{1}{(q^{-2};q^{-5})_\infty (q^{-3};q^{-5})_\infty},
\end{gather}
where we use the $q$-Pochhammer symbol in \eqref{eq:q_Pochhammer_inf}. 
\item When $R$ is the power series ring $\Fq\bbracks{t}$, giving a formula of $\Zhat_R(x)$ is equivalent to finding the number of nilpotent matrices over $\Fq$, which was given by Fine and Herstein \cite{fineherstein1958}. 
\item When $R=\Fq[u,v]$, the series $\Zhat_R(x)$ is the generating function evaluated by Feit and Fine \cite{feitfine1960pairs} to enumerate pairs of commuting matrices. 
\item When $R=\Fq\bbracks{u,v}$, the series $\Zhat_R(x)$ is the generating function evaluted by Fulman and Guralnick \cite{fulmanguralnick2018} to enumerate pairs of commuting nilpotent matrices. 
\end{itemize}

For any variety\footnote{separated scheme of finite type} $X$ over $\Fq$, we define the \emph{Cohen--Lenstra series} of $X$ over $\Fq$ as
\begin{equation}
\Zhat_{X/\Fq}(x):=\sum_M \dfrac{1}{\abs{\Aut M}} x^{\dim_\Fq H^0(X;M)},
\end{equation}
where $M$ ranges over all isomorphism classes of finite-length coherent sheaves over $X$, and $H^0(X;M)$ denotes the space of global sections of $M$. This generalizes the Cohen--Lenstra series over a ring, since $\Zhat_{R}(x)=\Zhat_{\Spec R}(x)$. If $p$ is a closed point of $X$, we define the \emph{local Cohen--Lenstra series} of $X$ over $p$ as
\begin{equation}
\Zhat_{X,p}(x):=\Zhat_{\O_{X,p}}(x)=\Zhat_{\Ohat_{X,p}}(x),
\end{equation}
where $\O_{X,p}$ is the local ring of $X$ at $p$ and $\Ohat_{X,p}$ is its completion. 

It is implicit in the work of Bryan and Morrison \cite{bryanmorrison2015motivic} (and references therein) that $\Zhat_{X}(x)$ and $\Zhat_{X,p}(x)$ are known if $X$ is a smooth curve or a smooth surface (Proposition \ref{prop:zhat_known}). The Cohen--Lenstra series satisfies an important ``Euler product'' property (Proposition \ref{prop:euler}): for any variety $X$, we have
\begin{equation}
\Zhat_X(x) = \prod_{p\in \cl(X)} \Zhat_{X,p}(x),
\end{equation}
where $\cl(X)$ denotes the set of closed points of $X$. In light of this property, the study of $\Zhat_X(x)$ is equivalent to the study of the local factors $\Zhat_{X,p}(x)$. When $X$ is a reduced curve or surface, the only unknown factors are those at singular points. We can view the local Cohen--Lenstra series as an invariant attached to the classification of singularities up to analytic isomorphism, so it is natural to wonder what $\Zhat_{X,p}(x)$ reveals about the geometry of $X$ at $p$.

\subsection{Main results}
The goal of this paper is to determine the properties of $\Zhat_{X,p}(x)$ where $X$ is a singular curve over $\Fq$ and $p$ is a \emph{nodal singularity}, namely, a singularity whose completed local ring is isomorphic to $\Fq\bbracks{u,v}/(uv)$. This is the first result about the local Cohen--Lenstra series of a singularity. We use the $q$-Pochhammer symbol
\begin{equation}\label{eq:q_Pochhammer_fin}
(a;q)_n:=(1-a)(1-qa)\dots (1-q^{n-1}a),
\end{equation}
\begin{equation}\label{eq:q_Pochhammer_inf}
(a;q)_\infty:=(1-a)(1-qa)(1-q^2 a)\dots.
\end{equation}

\begin{theorem}\label{thm:A}
Fix a prime power $q>1$ and let $R_q=\Fq\bbracks{u,v}/(uv)$. Then
\begin{enumerate}
\item The power series $\Zhat_{R_q}(x)$ in $x$ has a meromorphic continuation to all of $\C$.
\item The poles of the meromorphic continuation of $\Zhat_{R_q}(x)$ are precisely double poles at $x=q^i$, $i=1,2,\dots$. Moreover, the power series $\Zhat_{R_q}(x)$ admits a factorization
\begin{equation}\label{eq:factorize}
\Zhat_{R_q}(x) = \frac{1}{(xq^{-1};q^{-1})_\infty^2} H_q(x)
\end{equation}
where
\begin{equation}\label{eq:thmB-2}
H_q(x):=\sum_{k=0}^\infty \frac{x^{2k}q^{-k^2}}{(q^{-1};q^{-1})_k} (xq^{-k-1};q^{-1})_\infty.
\end{equation}
We point out that $H_q(x)$ is an entire power series in $x$ whenever $\abs{q}>1$. 

\item $\Zhat_{R_q}(1)=\dfrac{1}{(q^{-1};q^{-1})_\infty^2}$, and $\Zhat_{R_q}(-1)=\dfrac{1}{(-q^{-2};q^{-2})_\infty}$.
\end{enumerate}
\end{theorem}

Since $R_q=\Fq\bbracks{u,v}/(uv)$ is the completed local ring of a curve at a nodal singularity, the series $\Zhat_{R_q}(x)$ in Theorem \ref{thm:A} is precisely the local Cohen--Lenstra series of a nodal singularity. Theorem \ref{thm:A}(b) implies that the series $\Zhat_{R_q}(x)$ has a radius of convergence equal to $q$; the content of Theorem \ref{thm:A}(a)(b) lies in the description of $\Zhat_{R_q}(x)$ outside the domain of convergence. 
Theorem \ref{thm:A}(c) says that $\Zhat_{R_q}(\pm 1)$ have infinite product formulas while $\Zhat_{R_q}(x)$ is not known to have one beyond the factorization in \eqref{eq:factorize}. Theorem \ref{thm:A}(c) also has the following statistical interpretations. The value of $\Zhat_{R_q}(1)$ is the weighted count of finite-cardinality $R_q$-modules up to isomorphism, each weighted inversely by the size of the automorphism group. The numbers $(\Zhat_{R_q}(1)\pm \Zhat_{R_q}(-1))/2$ give the weighted counts of even- and odd-dimensional $R_q$-modules up to isomorphism, respectively. 

Let us briefly elaborate on some special features of Theorem \ref{thm:A}(b)(c) by looking at the smooth counterpart of Theorem \ref{thm:A}(b)(c). For comparison, let $S_q=\Fq\bbracks{t}$, the completed local ring of a curve at a smooth $\Fq$-point. Then we have
\begin{gather}
\Zhat_{S_q}(x)=\frac{1}{(xq^{-1};q^{-1})_\infty};\\
\Zhat_{S_q}(\pm 1)=\dfrac{1}{(\pm q^{-1};q^{-1})_\infty}.
\end{gather}

We note that $\Zhat_{R_q}(x)$ has a complicated factor $H_q(x)$ not present in $\Zhat_{S_q}(x)$. It is expected that $\Zhat_{R_q}(x)$ is more complicated than $\Zhat_{S_q}(x)$, since $\Zhat_{R_q}(x)$ and $\Zhat_{S_q}(x)$ encode counts of $R_q$- and $S_q$-modules, and the classification of $R_q$-modules (see \cite{bgs2009pairs} and its references) is much more complicated than the classically well-known classification of modules over $S_q=\Fq\bbracks{t}$. Less expected is the fact that $\Zhat_{R_q}(x)$ has a natural factorization at all. One may interpret $H_q(x)$ as a factor accounting for the nodal singularity. However, when specializing to $x=\pm 1$, the complicated factor in $\Zhat_{R_q}(x)$ seems to disappear, although the expression for $\Zhat_{R_q}(1)$ appears to be essentially different from $\Zhat_{R_q}(-1)$, unlike in the case of $\Zhat_{S_q}(\pm 1)$. 

Theorem \ref{thm:A} also implies that if $X$ is a reduced curve with only nodal singularities, and $\ttilde X$ is its resolution of singularity, then $\Zhat_X(x)/\Zhat_{\ttilde X}(x)$ is entire. This can be interpreted as that the Cohen--Lenstra series of a nodal singular curve, while being mysterious, is ``not too far'' from its smooth version.

Our proof of Theorem \ref{thm:A} depends on the following identity, which admits a proof by direct counting (Section \ref{sec:proof}). 

\begin{theorem}\label{thm:B}
As power series in $x$, we have
\begin{equation}\label{eq:thmB-1}
\displaystyle\sum_{n=0}^\infty \dfrac{\abs{\set{(A,B)\in \Mat_n(\Fq)\times\Mat_n(\Fq):AB=BA=0}}}{\abs{\GL_n(\Fq)}}x^n = \frac{1}{(x;q^{-1})_\infty^2} H_q(x),
\end{equation}
where $H_q(x)$ is defined in \eqref{eq:thmB-2}.
\end{theorem}

Theorem \ref{thm:B} leads to Theorem \ref{thm:A} by the general properties of the Cohen--Lenstra series  in Section \ref{sec:general}; the argument is given in Section \ref{sec:Hq}. 

The content of Theorem \ref{thm:B} lies not only in the enumeration of pairs of mutually annihilating matrices, but also in the unusual factorization identity \eqref{eq:thmB-1}. We point out that the left-hand side of \eqref{eq:thmB-1} is precisely $\Zhat_{\Fq[u,v]/(uv)}(x)$ (note the single bracket). We also point out that obtaining the specific expression of $H_q(x)$ in \eqref{eq:thmB-2} is the key to prove Theorem \ref{thm:A}. It is unknown if any part of Theorem \ref{thm:A} (even part (a) on the existence of a meromorphic continuation to $\C$) can be proved without knowing the exact formula of $H_q(x)$. 

Theorem \ref{thm:A} implies a formula that counts pairs of mutually annihilating \emph{nilpotent} matrices:
\begin{equation}\label{eq:nilpotent}
\displaystyle\sum_{n=0}^\infty \dfrac{\abs{\set{(A,B)\in \mathrm{Nilp}_n(\Fq)\times\mathrm{Nilp}_n(\Fq):AB=BA=0}}}{\abs{\GL_n(\Fq)}}x^n = \frac{1}{(xq^{-1};q^{-1})_\infty^2} H_q(x),
\end{equation}
where $\mathrm{Nilp}_n(\Fq)$ denotes the set of $n$ by $n$ nilpotent matrices over $\Fq$, and $H_q(x)$ is defined in \eqref{eq:thmB-2}. In particular, counting mutually annihilating nilpotent matrices can be viewed as a consequence of counting mutually annihilating matrices and the general properties of the Cohen--Lenstra series. After the posting of this paper, Fulman and Guralnick \cite{fulmanguralnick2022} found a direct proof of \eqref{eq:nilpotent} that does not require the enumeration of mutually annihilating matrices.

As an attempt to extend Theorem \ref{thm:A} to other curve singularities, we formulate the following questions:
\begin{question}\label{q:main}
Let $p$ be an $\Fq$-point of a reduced curve $X$.
\begin{enumerate}
\item Is it always true that $\Zhat_{X,p}(x)$ has a meromorphic continuation to all of $\C$?
\item If the answer to (a) is yes, is it true that the poles of $\Zhat_{X,p}(x)$ are given by the factor $(xq^{-1};q^{-1})_\infty^{-r(p)}$, where $r(p)$ is some numeric invariant attached to the pair $(X,p)$?
\item Do the special values $\Zhat_{X,p}(\pm 1)$ read the geometry of $X$ at $p$ in a meaningful way?
\end{enumerate}
\end{question}

A possibility that is compatible to all the known cases (i.e., smooth point and nodal singularity) is that $r(p)$ is the branching number of $X$ at $p$. In particular, we conjecture that
\begin{conjecture}\label{conj:main}
Let $p$ be an $\Fq$-point of a reduced curve $X$. Then $\Zhat_{X,p}(x)$ has a meromorphic continuation to all of $\C$ given by the factorization
\begin{equation}
\Zhat_{X,p}=\frac{1}{(xq^{-1};q^{-1})_\infty^r}H_{X,p}(x),
\end{equation}
where $r$ is the branching number of $X$ at $p$ and $H_{X,p}(x)$ is an entire power series. 
\end{conjecture}

An analogous generating series for the local Hilbert schemes of points supported on any reduced curve singularity $p$ was shown to be rational. Moreover, the denominator is a power whose exponent is the branching number at $p$; see \cite{brv2020motivic}. (Prior to that, the same fact was observed for planar singularities in \cite{maulikyun2013macdonald}, and was generalized to all Gorenstein singularities in \cite{goettscheshende2014refined}.) Conjecture \ref{conj:main} is also equivalent to a global restatement that if $\ttilde X$ is the resolution of singularity of any reduced curve $X$, then the quotient power series $\Zhat_X(x)/\Zhat_{\ttilde X}(x)$ is entire.

Finally, we point out that our proof of Theorem \ref{thm:B} requires nothing special about positive characteristic, and is ``parametric'' in the sense that the argument in Section \ref{subsec:count} essentially realizes the variety of mutually annihilating matrices as a disjoint union of certain strata, each of which is a fiber bundle. In particular, our proof implies the following identity as power series over the ring $\KVar\bbracks{\L^{-1}}$, where $\KVar$ is the Grothendieck ring of complex varieties and $\L=[\C^1]$ is the motive of the affine line:
\begin{equation}\label{eq:thm-motive}
\sum_{n=0}^\infty \dfrac{[\set{(A,B)\in \Mat_n(\C)^2:AB=BA=0}]}{\bracks{\GL_n(\C)}}x^n = \frac{1}{(x;\L^{-1})_\infty^2} \mathcal{H}(x),
\end{equation}
where
\begin{equation}
\mathcal{H}(x)=\sum_{k=0}^\infty \frac{x^{2k}\L^{-k^2}}{(\L^{-1};\L^{-1})_k} (x\L^{-k-1};\L^{-1})_\infty.
\end{equation}

Let $X=\set{uv=0}\subeq \C^2$ be the union of axes. The left-hand side of \eqref{eq:thm-motive} can be alternatively written as
\begin{equation}
\sum_{n=0}^\infty [\Coh_n(X)]x^n,
\end{equation}
where $[\Coh_n(X)]$ is the motive of the stack of coherent sheaves on $X$ of length $n$, in the sense of \cite{bryanmorrison2015motivic}. 

\subsection{Related work}
The coefficients of the Cohen--Lenstra series encode the point count of a wide family of varieties that arise as variants of the commuting variety. The commuting variety is the variety of pairs of commuting matrices, whose geometry was studied by Motzkin and Taussky \cite{motzkintaussky1955} and Gerstenhaber \cite{gerstenhaber1961}. Generalizations and variants of the commuting variety have been introduced and their geometry has been studied in both characteristic zero and positive characteristic; see \cite{baranovsky2001variety,chenlu2019,chenwang2020,cbs2002irreducible,richardson1979commuting}.   

One of the variants of the commuting variety is the variety of modules \cite{cbs2002irreducible}. Our Theorems \ref{thm:A} and \ref{thm:B}, in particular, give the point count of the varieties of modules over $\Fq\bbracks{u,v}/(uv)$ and $\Fq[u,v]/(uv)$, respectively. Their point count encodes statistical information (in the sense of Cohen--Lenstra \cite{cohenlenstra1984heuristics}) about the classification of finite-dimensional modules over $\Fq[u,v]/(uv)$ up to isomorphism. See \cite{bgs2009pairs, nrsb1975application, schroeder2004varieties} for studies of the structure of these varieties and the aforementioned classification problem. See also \cite{moschettiricolfi2018} for similar work about $\Fq[u,v]$. 

A stack quotient of the variety of $n$-dimensional modules by $\GL_n$ gives the moduli stack of length-$n$ coherent sheaves over a variety \cite{bryanmorrison2015motivic}. The Cohen--Lenstra series is closely related to the motive (in the Grothendieck ring of varieties) of this stack. Bryan and Morrison \cite{bryanmorrison2015motivic} reproved and refined the result of Feit and Fine about $\Zhat_{\Fq[u,v]}$ in terms of the motive of $\Coh_n(\C^2)$, the stack of length-$n$ coherent sheaves over $\C^2$. Having the notion of motive to play the role of point counting over a finite field, we shall focus on complex varieties for the rest of this section. 

While motivic knowledge about $\Coh_n(X)$ is mostly limited to the cases where $X$ is a smooth curve or a smooth surface, much more is known motivically about its framed\footnote{The terminology will be clear in the next paragraph, and it is in line with the notion of framed and unframed quiver varieties following Nakajima \cite{nakajima1999hilbert}.} counterpart, namely, the Hilbert scheme of points on $X$. The motive of the Hilbert scheme is known for any smooth surface \cite{cheah1994cohomology, goettsche1990betti, goettsche2001motive}, for any reduced singular curve \cite{brv2020motivic} (where an analog of Conjecture \ref{conj:main} holds), and even for some singular surfaces \cite{gns2017euler}. Conjecture \ref{conj:main} can be viewed as an unframed analog of \cite{brv2020motivic}, and one of our motivations to study Conjecture \ref{conj:main} is to understand how the Hilbert scheme and the stack of coherent sheaves are related motivically.

There are rich connections between the Hilbert scheme and the the stack of coherent sheaves. First, the Hilbert scheme $\Hilb^n(X)$ of a (possibly singular) complex variety $X$ parametrizes length-$n$ $\O_X$-modules $M$ together with a ``framing'' $\O_X\onto M$. Forgetting the framing gives a map from $\Hilb^n(X)$ to $\Coh_n(X)$. Such a map fits in the context of Nakajima quiver varieties \cite{nakajima1999hilbert}, and sometimes one can prove Kirwan surjectivity results \cite{mcgertynevins2018}. 

Another potential analog of \cite{brv2020motivic} concerns the motive of the Quot scheme. The Quot scheme of $d$-framed modules over $X$ parametrizes finite-length $\O_X$-modules $M$ together with a $d$-framing $\O_X^d\onto M$, where $d\geq 1$. The Hilbert scheme is thus the Quot scheme for $d=1$. A formula for the motive of the Quot schemes over a smooth curve is given in \cite{monavariricolfi2022}. The same work generalizes the formula to the nested Quot schemes. Studying the motive of the (nested) Quot schemes over singular curves appears to be of interest parallel to Conjecture \ref{conj:main}. Our forthcoming joint work \cite{huangjiang2022b} suggests that the motive of the stack of coherent sheaves over a singular curve can be understood from the motive of the Quot schemes for all $d\geq 1$, using the natural map from the Quot scheme to the stack of coherent sheaves defined by forgetting the framing.  

Apart from the motive in the Grothendieck ring of varieties, the Borel--Moore homology and the $K$-theory are also potentially interesting perspectives to look at the Hilbert scheme and the stack of coherent sheaves, especially for singular curves. If $X=\C^2$, then the Borel--Moore homology (or the $K$-theory) of $\Hilb^n(X)$ are connected to that of $\Coh_n(X)$; see \cite{schiffmannvasserot2013cherednik, schiffmannvasserot2013elliptic}. Of course, the Hilbert scheme and the stack of coherent sheaves over a singular curve are singular, and the Borel--Moore homology and the $K$-theory, unlike the motive in the Grothendieck ring of varieties, are sensitive to singularities and noncompactness. However, we point out that while $\Hilb^n(\C^2)$ is smooth, $\Coh_n(\C^2)$ is not.

\subsection{Organization of the paper}
In Section \ref{sec:preliminaries}, we give some preliminaries about partitions and $q$-series that will be used in the proof of Theorem \ref{thm:B}, given in Section \ref{sec:proof}. In Section \ref{sec:general}, we give a self-contained introduction to known properties of the Cohen--Lenstra series, part of which will be used to prove Theorem \ref{thm:A}. In Section \ref{sec:Hq}, we make some elementary observations about the series $H_q(x)$ appearing in the main theorems, and use them to finish the proof of Theorem \ref{thm:A}. In Section \ref{sec:further}, we discuss a possible connection between $H_q(x)$ and the partial theta function. 

\section{Preliminaries}\label{sec:preliminaries}
A \emph{partition} $\lambda$ is a finite nonincreasing sequence of positive integers $(\lambda_1,\dots,\lambda_\ell)$, each of which is called a \emph{part} of $\lambda$. The \emph{length} of $\lambda$ is the number of parts in $\lambda$, denoted $\ell(\lambda)$. The \emph{size} of $\lambda$ is $\abs{\lambda}:=\sum_i \lambda_i$. We denote by $a_i(\lambda)$ the number of parts of $\lambda$ of size $i$, so we can write down a partition as
\begin{equation}\lambda=a_1(\lambda)\cdot [1] + a_2(\lambda)\cdot [2]+\dots
\end{equation}

The \emph{Young diagram} of $\lambda$ follows the convention such that it has $\lambda_1$ boxes in the top row, and $\ell(\lambda)$ boxes in the leftmost column. We will often refer to a partition by its Young diagram. The \emph{conjugate} (or \emph{transpose}) partition of $\lambda$ is the partition, denoted $\lambda'$, whose Young diagram is the transpose of the Young diagram of $\lambda$. 

The (first) \emph{Durfee square} of $\lambda$ is the largest square that fits the top-left corner of its Young diagram. For $i\geq 1$, the $(i+1)$-st Durfee square is the Durfee square of the part of $\lambda$ below the $i$-th Durfee square. We denote the sidelength of the $i$-th Durfee square by $\sigma_i(\lambda)$, and define the \emph{Durfee partition} as
\begin{equation}
\Durf(\lambda):=(\sigma_1(\lambda),\sigma_2(\lambda),\dots).\end{equation}

Recall the definition of the $q$-\emph{Pochhammer symbol} in \eqref{eq:q_Pochhammer_fin} and \eqref{eq:q_Pochhammer_inf}. The $q$-\emph{binomial coefficient} is defined as
\begin{equation}
{n \brack k}_q:=\frac{(q;q)_n}{(q;q)_k (q;q)_{n-k}}.
\end{equation}

\section{Proof of Theorem \ref{thm:B}}\label{sec:proof}
We define $\Zhat(x):=\sum_{n=0}^\infty \dfrac{\abs{\set{(A,B)\in \Mat_n(\Fq)\times \Mat_n(\Fq):AB=BA=0}}}{\abs{\GL_n(\Fq)}}x^n$ and compute it in two steps. Though not needed in the proof, it is worth noting that $\Zhat(x)=\Zhat_{\Fq[u,v]/(uv)}$.

\subsection{Counting pairs of mutually annihilating matrices}\label{subsec:count}
Fix $n$. We count the number of pairs $(A,B)$ of $n\times n$ matrices such that $AB=BA=0$. 

First, we fix $A$ and let $0\leq k\leq n$ be the nullity of $A$ (so the rank of $A$ is $n-k$). Let $\im A=V$, $\ker A=W$, then $\dim V=n-k$, $\dim W=k$. We have
\begin{gather}
AB=0\iff A(\im B)=0\iff \im B\subeq W; \\
BA=0\iff B(\im A)=0\iff \ker B\supeq V.
\end{gather}

Hence, choosing $B$ that mutually annihilates $A$ is equivalent to picking a linear map from $\Fq^n/V \to W$. Since $\dim \Fq^n/V=\dim W=k$, there are $q^{k^2}$ choices of $B$. Notice that this number depends only on the rank of $A$.

It is a standard fact that the number of $n\times n$ $\Fq$-matrices of nullity $k$ is
\begin{equation}
{n \brack k}_q (q^n-1)(q^n-q)\dots (q^n-q^{n-k-1}).
\end{equation}
(As one of the proofs, we first choose an $(n-k)$-dimensional subspace $V\subeq \Fq^n$ as the image, and then choose a surjection $\Fq^n\to V$. The former has ${n \brack k}_q$ choices, and the latter has $(q^n-1)(q^n-q)\dots (q^n-q^{n-k-1})$ choices.)

Recalling that $\abs{\GL_n(\Fq)}=(q^n-1)(q^n-q)\dots (q^n-q^{n-1})$, it follows (after simplication) that
\begin{align}
\Zhat(x)&=\sum_{n=0}^\infty \sum_{k=0}^n \dfrac{1}{\abs{\GL_n(\Fq)}} q^{k^2} {n \brack k}_q (q^n-1)(q^n-q)\dots (q^n-q^{n-k-1}) x^n\\
&=\sum_{n=0}^\infty \sum_{k=0}^n \frac{{n\brack k}_t}{(t;t)_k} x^n,
\end{align}
where $t=q^{-1}$. 

\subsection{Factorization of $\Zhat(x)$} 
We collect some standard $q$-series identities that will be used in the factorization. 

\begin{proposition}[\cite{andrewspartitions}]\label{prop:q_identities}
As formal power series in $t$ \textup{(}and $x$ if applicable\textup{)}, we have
\begin{enumerate}
\item $\sum_{\ell(\lambda)\leq k} t^{\abs{\lambda}}=\dfrac{1}{(t;t)_k}$, where the sum is over all partitions with at most $k$ parts.
\item $\sum_{\lambda\subeq (n-k)\times k} t^{\abs{\lambda}} = {n\brack k}_t$. The notation here means that the sum is over all partitions whose Young diagram fits inside an $(n-k)\times k$ rectangle.
\item $\sum_{n=0}^\infty \dfrac{x^n}{(t;t)_n} = \dfrac{1}{(x;t)_\infty}$.
\item A \emph{partition with zeros} is a nonincreasing sequence of finitely many nonnegative integers $(\lambda_1,\lambda_2,\dots,\lambda_\ell)$, whose length is defined as $\ell$. Then we have the identity
\begin{equation}\sum_{\substack{\lambda \wzero \\ \lambda_1\leq k}} t^{\abs{\lambda}} x^{\ell(\lambda)} = \frac{1}{(1-x)(1-tx)\dots (1-t^kx)},\end{equation}
where the sum is over all partitions with zeros whose parts have sizes at most $k$.
\end{enumerate}
\begin{remark}
The set of partitions with zeros of length $n$ is in one-to-one correspondence with the set of usual partitions having at most $n$ parts, making the notion of partitions with zeros not strictly necessary. However, this notion is preferred here to emphasize that $n$ is remembered as part of the data. 
\end{remark}
\end{proposition}
\begin{proof}
Part (b) is \cite[Theorem 3.1]{andrewspartitions}. Parts (a)(c)(d) follow from elementary methods introduced in \cite[\S 1]{andrewspartitions}. To be precise, part (a) follows from \cite[Theorems 1.1 and 1.4]{andrewspartitions}. Part (d) follows from
\begin{itemize}
\item identifying each $\lambda$ in the sum with a tuple $(a_0,\dots,a_k)$, where $a_i$ is the number of $i$'s in $\lambda$; and
\item noting that $\abs{\lambda}=a_1+2a_2+\dots+ka_k$ and $\ell(\lambda)=a_0+\dots+a_k$. 
\end{itemize} 

Part (c) is due to Euler; see \cite[Corollary 2.2]{andrewspartitions}. An elementary proof involves recognizing the left-hand side as
\begin{equation}
\sum_{\lambda \wzero} t^{\abs{\lambda}} x^{\ell(\lambda)},
\end{equation}
using the formula of part (a), and then showing that it equals the right-hand side reusing the argument for part (d). 
\end{proof}

\begin{lemma}
We have
\begin{equation}
\Zhat(x)=\sum_{n=0}^\infty \sum_{k=0}^n \frac{{n\brack k}_t}{(t;t)_k} x^n=\sum_{\lambda\wzero} t^{\abs{\lambda}-\sigma_1(\lambda)^2}x^{\ell(\lambda)}
\end{equation}
\end{lemma}
\begin{proof}
A partition $\lambda$ whose Durfee square has sidelength $k$ can be reconstructed uniquely with a partition $\lambda^{(1)}$ such that $\ell(\lambda^{(1)})\leq k$ (to be put to the right of the Durfee square) and a partition $\lambda'$ such that $\lambda'_1\leq k$ (to be put below the Durfee square).

By Proposition \ref{prop:q_identities}(1)(2), we have
\begin{align}
\sum_{n=0}^\infty \sum_{k=0}^n \frac{{n\brack k}_t}{(t;t)_k} x^n &= \sum_{n=0}^\infty x^n \sum_{k=0}^n \parens*{\sum_{\lambda^{(1)}\subeq k\times \infty} t^{\abs{\lambda^{(1)}}} \sum_{\lambda'_1\subeq (n-k)\times k} t^{\abs{\lambda'}} } \\
&=\sum_{n=0}^\infty x^n\sum_{\lambda\subeq n\times \infty} t^{\abs{\lambda}-\sigma_1(\lambda)^2} \label{eq:line1-2}\\
&=\sum_{n=0}^\infty x^n \sum_{\substack{\lambda \wzero\\\ell(\lambda)=n}} t^{\abs{\lambda}-\sigma_1(\lambda)^2} \label{eq:line1-3}\\
&=\sum_{\lambda \wzero} t^{\abs{\lambda}-\sigma_1(\lambda)^2} x^{\ell(\lambda)}.
\end{align}

Here, the line \eqref{eq:line1-2} is because $\lambda^{(1)}$ and $\lambda'$ consist of the part of $\lambda$ outside the Durfee square. Note that $\lambda\subeq n\times \infty$ means that $\lambda$ fits inside the $n$ (rows) $\times$ $\infty$ (columns) rectangle, which is equivalent to saying $\ell(\lambda)\leq n$. The line \eqref{eq:line1-3} is because specifying a partition whose length is at most $n$ is equivalent to specifying a partition with zeros whose length is exactly $n$. 
\end{proof}

To complete the factorization, we reconstruct $\lambda$ with zeros using the first two Durfee squares. Pick $k\geq l\geq 0$. Let $\lambda^{(1)}\subeq k\times \infty$ and $\lambda^{(2)}\subeq l\times (k-l)$ be usual partitions, and $\lambda''\subeq \infty \times l$ be a partition with zeros. Then we have a bijection
\begin{equation}
\set{(\lambda''\wzero,\lambda^{(1)},\lambda^{(2)})\text{ as above}}\to\set{\lambda \wzero: \sigma_1(\lambda)=k,\sigma_2(\lambda)=l}
\end{equation}
by putting $\lambda^{(1)}$ to the right of the first Durfee square, $\lambda^{(2)}$ to the right of the second Durfee square, and $\lambda''$ below the second Durfee square. 

We have
\begin{align}
\Zhat(x)&=\sum_{\lambda \wzero} t^{\abs{\lambda}-\sigma_1(\lambda)^2} x^{\ell(\lambda)}\\
&=\sum_{\substack{k\geq l\geq 0\\ \lambda'',\lambda^{(1)},\lambda^{(2)}}} t^{\abs{\lambda^{(1)}}+\abs{\lambda^{(2)}}+\abs{\lambda''}+l^2}x^{k+l+\ell(\lambda'')} \\
&=\sum_{k\geq l\geq 0} t^{l^2}x^{l+k} 
\parens*{
\sum_{\lambda^{(1)}\subeq k\times \infty} t^{\abs{\lambda^{(1)}}} 
\sum_{\lambda^{(2)}\subeq k\times \infty} t^{\abs{\lambda^{(2)}}} 
\sum_{\substack{\lambda'' \wzero \\ \lambda''\subeq \infty\times l}} t^{\abs{\lambda''}} x^{\ell(\lambda'')}
} \\
&=\sum_{k\geq l\geq 0} t^{l^2}x^{l+k} \frac{1}{(t;t)_k}{k\brack l}_t \frac{1}{(1-x)(1-tx)\dots (1-t^l x)},\label{eq:line2-4}
\end{align}
where the line \eqref{eq:line2-4} uses Proposition \ref{prop:q_identities}(1)(2)(4), in that order. 

Observe that
\begin{equation}\label{eq:line3}
\frac{1}{(t;t)_k}{k \brack l}_t = \frac{1}{(t;t)_k} \frac{(t;t)_k}{(t;t)_l (t;t)_{k-l}} = \frac{1}{(t;t)_l (t;t)_{k-l}}.
\end{equation}

Letting $b=k-l$, we have
\begin{align}
\Zhat(x)&=\sum_{k\geq l\geq 0} t^{l^2}x^{l+k} \frac{1}{(t;t)_k}{k\brack l}_t \frac{1}{(1-x)(1-tx)\dots (1-t^l x)}\\
&=\sum_{b,l\geq 0} t^{l^2}x^{l+(b+l)} \frac{1}{(t;t)_l (t;t)_b} \frac{1}{(1-x)(1-tx)\dots (1-t^l x)} \\
&=\parens*{\sum_{b=0}^\infty \frac{x^b}{(t;t)_b}}
\parens*{ \sum_{l=0}^\infty t^{l^2}x^{2l} \frac{1}{(t;t)_l} \frac{1}{(1-x)(1-tx)\dots (1-t^l x)} }\\
&=\frac{1}{(x;t)_\infty} \sum_{l=0}^\infty t^{l^2}x^{2l} \frac{1}{(t;t)_l} \frac{1}{(1-x)(1-tx)\dots (1-t^l x)},\label{eq:line4-4}
\end{align}
where the line \eqref{eq:line4-4} follows from Proposition \ref{prop:q_identities}(3). 

We remark that the key reason why this factorization works is that \eqref{eq:line3} does not depend explicitly on $k$ after simplification. 

Finally, recalling that $(x;t)_\infty = (1-x)(1-tx)(1-t^2 x)\dots$, we get
\begin{align}
\Zhat(x)&=\frac{1}{(x;t)_\infty} \sum_{l=0}^\infty t^{l^2}x^{2l} \frac{1}{(t;t)_l} \frac{1}{(1-x)(1-tx)\dots (1-t^l x)}\\
&=\frac{1}{(x;t)_\infty^2} \sum_{l=0}^\infty t^{l^2}x^{2l} \frac{1}{(t;t)_l} (1-t^{l+1}x)(1-t^{l+2}x)\dots \\
&=: \frac{1}{(x;t)_\infty^2} H_q(x),
\end{align}
which finishes the proof of Theorem \ref{thm:B}.

\section{General properties of the Cohen--Lenstra series}\label{sec:general}
We give a self-contained introduction about the well-known properties of the Cohen--Lenstra series. These properties are implicit in \cite{bryanmorrison2015motivic,guseinzade2004power} from the motivic point of view, while in this introduction, we restrict our attention to counting over finite fields. We point out that the argument in Proposition \ref{prop:zhat_known} that uses Proposition \ref{prop:euler} is essentially the use of ``power structures'' in \cite{bryanmorrison2015motivic,guseinzade2004power}. 

For the purpose of proving the main results of this paper, only Propositions \ref{prop:euler}, \ref{prop:comparison}(a)(b) and \ref{prop:zhat_known}(a) are needed. 

Let $R$ be an algebra over $\Fq$ with only finite quotient fields, and let $X$ be a variety over $\Fq$. Recall the definitions
\begin{equation}
\Zhat_{R}(x):=\sum_M \dfrac{1}{\abs{\Aut M}} x^{\dim_\Fq M},
\end{equation}
where $M$ ranges over all isomorphism classes of finite-cardinality $R$-modules, and
\begin{equation}
\Zhat_{X}(x):=\sum_M \dfrac{1}{\abs{\Aut M}} x^{\dim_\Fq H^0(X;M)},
\end{equation}
where $M$ ranges over all isomorphism classes of finite-length coherent sheaves over $X$, and $H^0(X;M)$ denotes the space of global sections of $M$. We denote both $\dim_\Fq M$ and $\dim_\Fq H^0(X;M)$ by $\deg M$. We also recall the local Cohen--Lenstra series for a closed point $p$ of $X$:
\begin{equation}
\Zhat_{X,p}(x):=\Zhat_{\O_{X,p}}(x).
\end{equation}

We state some basic properties.

\begin{proposition}
Let $R,X,p$ be as above. Then
\begin{enumerate}
\item $\Zhat_{R}(x)=\Zhat_{\Spec R}(x)$.
\item $\Zhat_{\O_{X,p}}(x)=\Zhat_{\Ohat_{X,p}}(x)$.
\item We have
\begin{equation}
\Zhat_{X,p}(x):=\sum_{M_p} \dfrac{1}{\abs{\Aut M_p}} x^{\deg M_p},
\end{equation}
where $M_p$ ranges over all isomorphism classes of finite-length coherent sheaves over $X$ that are supported at $p$. 
\end{enumerate}
\end{proposition}
\begin{proof}~
\begin{enumerate}
\item This follows from the standard correspondence between modules over $R$ and quasicoherent sheaves over $\Spec R$.
\item This follows from the elementary fact that the classification of finite-length modules over a Noetherian local ring is the same as the classification of finite-length modules over its completion.
\item A coherent sheaf is supported at $p$ is determined by it stalk at $p$, thus corresponds to a module over $\O_{X,p}$. 
\end{enumerate}
\end{proof}

\begin{proposition}[Euler product] \label{prop:euler}
Let $X$ be a variety over $\Fq$. Then
\begin{equation}
\Zhat_X(x)=\prod_{p\in \cl(X)} \Zhat_{X,p}(x),
\end{equation}
where $\cl(X)$ is the set of closed points in $X$.
\end{proposition}
\begin{proof}
For every finite-length coherent sheaf $M$ over $X$, we have a unique decomposition $M=\bigoplus_{p\in \cl(X)} M_p$ into finite-length coherent sheaves $M_p$ supported at $p$, with all but finitely many $M_p$'s being zero. For closed points $p\neq q$ and sheaves $M_p, M_q$ supported on $p,q$ respectively, we have
\begin{equation}
\Hom_X(M_p,M_q)=0.
\end{equation}

It follows that
\begin{equation}\Aut(M_{p_1}\oplus \dots \oplus M_{p_r}) \cong \Aut(M_{p_1})\times \dots \times \Aut(M_{p_r}).
\end{equation}

As a consequence,
\begin{align}
\Zhat_X(x)&=\sum_{M} \frac{1}{\abs{\Aut M}}x^{\deg M}\\
&=\sum_{(M_p:\,p\in \cl(X))} \frac{1}{\abs{\prod_{p}\Aut M_p}} x^{\sum_p \deg M_p}\\
&=\prod_{p\in \cl(X)} \sum_{M_p} \frac{1}{\abs{\Aut M_p}}x^{\deg M_p}\\
&=\prod_{p\in \cl(X)} \Zhat_{X,p}(x).
\end{align}
\end{proof}

For any subvariety $Y$ of $X$, if we set
\begin{equation}\Zhat_{X,Y}(x):=\prod_{p\in \cl(Y)}\Zhat_{X,p}(x)=\sum_{\supp M\subeq Y} \dfrac{1}{\abs{\Aut M}}x^{\deg M},\end{equation}
then the Euler product gives
\begin{equation}\label{eq:motivic}
\Zhat_{X}(x)=\Zhat_{U}(x)\cdot \Zhat_{X,Z}(x)
\end{equation}
for any open subvariety $U\subeq X$ and closed subvariety $Z\subeq X$ with $X\setminus Z=U$. This uses the fact that $\O_{U,p}=\O_{X,p}$ for all $p\in \cl(U)$, so that 
\begin{equation}
\Zhat_{U}(x)=\Zhat_{X,U}(x)\,\text{ for open $U\subeq X$.}
\end{equation}

As a warning, $\Zhat_{X,Z}(x)$ is not equal to $\Zhat_Z(x)$, because $\O_{Z,p}$ and $\O_{X,p}$ are not isomorphic in general. Thus, the $\Zhat$ construction is not motivic in the sense that $\Zhat_{X}(x)\neq \Zhat_{U}(x)\cdot \Zhat_{Z}(x)$.

The following relates the Cohen--Lenstra zeta function to the variety of modules, or its nilpotent variant. 

\begin{proposition}\label{prop:comparison}
Let $R=\dfrac{\Fq[t_1,\dots,t_m]}{(f_1,...,f_r)}$, where $t_1,\dots,t_m$ are indeterminates and $f_1,\dots,f_r$ are polynomials in $t_1,\dots,t_m$. Then
\begin{enumerate}
\item We have
\begin{equation}
\Zhat_R(x)= \sum_{n=0}^\infty \frac{\abs{M_n}}{\abs{\GL_n(\Fq)}} x^n,
\end{equation}
where
\begin{equation}
M_n := \set*{ (A_1,\dots,A_m)\; \left \lvert \;
\begin{gathered}
A_i\in \Mat_n(\Fq),A_iA_j=A_jA_i\\
f_s(A_1,\dots,A_m)=0\textup{ for }1\leq s\leq r
\end{gathered}
\right.
}
\end{equation}
\item If $f_1,\dots,f_r$ all vanish at the origin $0$, let $p=(t_1,\dots,t_m)R$ be the maximal ideal corresponding to $0\in X:=\Spec R$, then
\begin{equation}\Zhat_{R_p}(x)=\Zhat_{X,0}(x)= \sum_{n=0}^\infty \frac{\abs{N_n}}{\abs{\GL_n(\Fq)}} x^n,
\end{equation}
where
\begin{equation}
N_n := \set*{ (A_1,\dots,A_m)\; \left \lvert \;
\begin{gathered}
A_i\in \Nilp_n(\Fq),A_iA_j=A_jA_i\\
f_s(A_1,\dots,A_m)=0\textup{ for }1\leq s\leq r
\end{gathered}
\right.
}
\end{equation}
\item More generally, let $I\subeq J\subeq \Fq[t_1,\dots,t_m]$ be two ideals, and $Z = \Spec \Fq[t_1,\dots,t_m]/J$ $\subeq X=\Spec \Fq[t_1,\dots,t_m]/I$. Then 
\begin{equation}
\Zhat_{X,Z}(x)= \sum_{n=0}^\infty \frac{\abs{N_n}}{\abs{\GL_n(\Fq)}} x^n,
\end{equation}
where
\begin{equation}
N_n := \set*{ A=(A_1,\dots,A_m)\; \left \lvert \;
\begin{gathered}
A_i\in \Mat_n(\Fq), A_iA_j=A_jA_i\\
f(A)=0\textup{ for }f\in I, g(A)\in \Nilp_n(\Fq)\textup{ for }g\in J
\end{gathered}
\right.
}
\end{equation}
\end{enumerate}
\end{proposition}

\begin{proof}~
\begin{enumerate}
\item Fix $n$ and consider an $\Fq$-vector space $V$ of dimension $n$. Giving $V$ a structure of an $R$-module is equivalent to specifying the actions of $t_1,\dots,t_m$ on $V$ as linear endomorphisms $A_1,\dots,A_m$, under the constraints 
\begin{equation}
A_iA_j=A_jA_i
\end{equation} and 
\begin{equation}
f_s(A_1,\dots,A_m)=0\text{ for }1\leq s \leq r.
\end{equation}
We denote by $(V;A_1,\dots,A_m)$ the $R$-module specified by the data above. Note that the constraints are satisfied if and only if $(A_1,\dots,A_m)\in M_n$.

Consider the action of $GL_n(\Fq)$ on the set $M_n$ by simultaneous conjugation:
\begin{equation}
g\cdot (A_1,\dots,A_m) := (gA_1g^{-1},\dots,gA_mg^{-1}).
\end{equation}

Given two $R$-modules $M=(V;A_1,\dots,A_m)$ and $M'=(V';A'_1,\dots,A'_m)$, an $R$-linear map $M\to M'$ is nothing but an $\Fq$-linear map $B: V\to V'$ such that $B\circ A_i=A'_i\circ B$ for all $i$. It follows that $(\Fq^n;A_1,\dots,A_m)$ and $(\Fq^n;A'_1,\dots,A'_m)$ are isomorphic as $R$-modules precisely if $(A_1,\dots,A_m)$ and $(A'_1,\dots,A'_m)$ are in the same orbit. Moreover, $g:\Fq^n\to \Fq^n$ gives an automorphism of $(\Fq^n;A_1,\dots,A_m)$ as an $R$-module if and only if $g$ fixes $(A_1,\dots,A_m)$. Denoting by $O_M$ the orbit corresponding to an $R$-module $M$ with $\dim_{\Fq} M=n$, the orbit-stabilizer theorem gives
\begin{equation}\sum_{\deg M=n}\frac{1}{\abs{\Aut M}} = \sum_{M} \frac{\abs{O_M}}{\abs{\GL_n(\Fq)}} = \frac{\abs{M_n}}{\abs{\GL_n(\Fq)}},\end{equation}
where the first and second sums are over the isomorphism classes of $R$-modules of degree $M$. 

It follows that
\begin{equation}
\Zhat_R(x)=\sum_{n=0}^\infty \sum_{\deg M=n} \frac{1}{\abs{\Aut M}} x^n = \sum_{n=0}^\infty \frac{\abs{M_n}}{\abs{\GL_n(\Fq)}} x^n.
\end{equation}

\item The proof is exactly the same, modulo the following observation:

\emph{The category of finite-length $R_p$-modules is a full subcatgory of finite-length $R$-modules consisting of those annihilated by some power of $p$.}

An $R$-module $(V;A_1,\dots,A_m)$ is annihilated by a power of $p=(t_1,\dots,t_m)R$ if and only if a power of $A_i$ annihilates $V$ for all $i$. This is equivalent to requiring that all $A_i$ are nilpotent. 

\item Let $R=\Spec \Fq[t_1,\dots,t_m]/I$ and $S=\Spec \Fq[t_1,\dots,t_m]/J$. It suffices to prove the following claim:

\emph{A finite-length $R$-module $M$ is supported on $Z=\Spec S$ if and only if $J^n M=0$ for some $n$.}

To prove the claim, assume $J^n M=0$. Consider any maximal ideal $p$ of $R$ that corresponds to a closed point in $X\setminus Z$, then $p$ does not contain $J$. So there exists $u\in J$ such that $u\notin p$. Since $u^n M=0$ and $u$ is invertible in $R_p$, the localization $M_p$ of $M$ at $p$ is zero. This shows that that $M$ is supported on $Z$.

Conversely, assume $M$ is supported on $Z$. Then there are maximal ideals $p_1,\dots,p_h$ corresponding to closed points in $Z$ such that
\begin{equation}
M=\bigoplus_{i=1}^h M_{p_i},
\end{equation}
where $M_{p_i}$ is the localization of $M$ at $p_i$. Note that $p_i\supeq J$. Since $M_{p_i}$ is a finite-length module over the local ring $R_{p_i}$, there is $n_i$ such that $p_i^{n_i} M_{p_i}=0$. It follows that $J^{n_i}M_{p_i}=0$. Taking $n=\max\set{n_1,\dots,n_h}$, we get $J^n M=\bigoplus_{i=1}^h J^n M_{p_i}=0$, proving the claim.
\end{enumerate}
\end{proof}

The following two corollaries translate facts about matrix enumeration into formulas for Cohen--Lenstra series.

\begin{corollary}
\label{cor:A1star}
The Cohen--Lenstra series for the punctured line $\A^1\setminus 0$ is
\begin{equation}
\Zhat_{\A^1\setminus 0}(x)=\Zhat_{\Fq[t,t^{-1}]}(x)=\frac{1}{1-x}.
\end{equation}
\end{corollary}
\begin{proof}
To apply Proposition \ref{prop:comparison}(a), we need to rewrite the coordinate ring as $R:=\Fq[u,v]/(uv-1)$. The corresponding matrix counting problem concerns
\begin{equation}
M_n:=\set{(A,B)\in \Mat_n(\Fq)^2:AB=BA=I_{n\times n}}.
\end{equation}

Counting $M_n$ amounts to counting invertible $n\times n$ matrices. Therefore, we have
\begin{equation}
\Zhat_R(x)=\sum_{n=0}^\infty \frac{\abs{\GL_n(\Fq)}}{\abs{\GL_n(\Fq)}} x^n = \frac{1}{1-x}.
\end{equation}
\end{proof}

\begin{corollary}\label{cor:zhat}
We have
\begin{gather}
\Zhat_{\Fq\bbracks{t}}(x)=\prod_{i=1}^\infty (1-q^{-i}x)^{-1}, \label{eq:result1}\\
\Zhat_{\Fq\bbracks{u,v}}(x)=\prod_{i,j\geq 1}(1-q^{-j}x^i)^{-1}.
\end{gather}
\end{corollary}
\begin{proof}
Given Proposition \ref{prop:comparison}(b), the two formulas follow from the matrix counting results of Fine and Herstein \cite{fineherstein1958} and Fulman and Guralnick  \cite{fulmanguralnick2018}, respectively. 
\end{proof}

\begin{proposition} \label{prop:zhat_known}

Let $Z_X(x)$ be the Hasse--Weil zeta function of $X$ \tu{\cite[p.\ 449]{hartshorneAG}}. Then
\begin{enumerate}
\item If $X$ is a smooth curve over $\Fq$, then
\begin{equation}
\Zhat_X(x) = \prod_{i=1}^\infty Z_X(q^{-i}x) \in \C\bbracks{x}.
\end{equation}
\item If $X$ is a smooth surface over $\Fq$, then
\begin{equation}
\Zhat_X(x) = \prod_{i,j\geq 1} Z_X(q^{-j}x^i) \in \C\bbracks{x}.
\end{equation}
\end{enumerate} 
\end{proposition}
\begin{proof}~`
\begin{enumerate}
\item The Euler product (Proposition \ref{prop:euler}) gives
\begin{equation}
\Zhat_X(x)=\prod_{p\in \cl(X)} \Zhat_{X,p}(x)=\prod_{p\in \cl(X)} \Zhat_{\Ohat_{X,p}}(x).
\end{equation}

Since $X$ is a smooth curve, $\Ohat_{X,p}$ is a complete regular local ring of dimension one. By the Cohen structure theorem, $\Ohat_{X,p}\cong \kappa_p\bbracks{t}$, where $\kappa_p$ is the residue field of $\Ohat_{X,p}$. By Corollary \ref{cor:zhat},
\begin{equation}
\Zhat_{\kappa_p\bbracks{t}/\kappa_p}(x)=\sum_{M/\kappa_p\bbracks{t}} \frac{1}{\abs{\Aut M}}x^{\dim_{\kappa_p} M} = \prod_{i=1}^\infty (1-q^{-i\deg p}x)^{-1},
\end{equation}
where $\deg p$ is the degree of the field extension $[\kappa_p:\Fq]$, so that $q^{\deg p}$ is the cardinality of $\kappa_p$. 

Noting that $\dim_{\Fq} M = (\deg p)\dim_{\kappa_p} M$ for any $\kappa_p$-vector space $M$, we have
\begin{align}
\Zhat_{\Ohat_{X,p}}(x) &= \sum_{M/\kappa_p\bbracks{t}} \frac{1}{\abs{\Aut M}}x^{\dim_{\Fq} M}\\
&=\sum_{M/\kappa_p\bbracks{t}} \frac{1}{\abs{\Aut M}}(x^{\deg p})^{\dim_{\kappa_p} M} \\
&=\prod_{i=1}^\infty (1-q^{-i\deg p}x^{\deg p})^{-1}.
\end{align}

Recalling the Euler product of the Hasse--Weil zeta function
\begin{equation}
Z_X(x)=\prod_{p\in \cl(X)}(1-x^{\deg p})^{-1},
\end{equation}
we get
\begin{align}
\Zhat_X(x) &= \prod_{p\in \cl(X)} \prod_{i=1}^\infty (1-q^{-i\deg p}x^{\deg p})^{-1} \\
&=\prod_{i=1}^\infty \prod_{p\in \cl(X)} (1-(q^{-i}x)^{\deg p})^{-1} \\
&=\prod_{i=1}^\infty Z_X(q^{-i}x).
\end{align}

\item By the Cohen structure theorem, for any closed point $p$ on a smooth surface $p$, we have $\Ohat_{X,p}\cong \kappa_p\bbracks{u,v}$. Applying the same argument to the corresponding formula in Corollary \ref{cor:zhat}, we have
\begin{align}
\Zhat_{\Ohat_{X,p}}(x) &= \Zhat_{\kappa_p\bbracks{u,v}/\kappa_p}(x^{\deg p})\\
& = \ev*{\parens*{\prod_{i,j\geq 1}(1-(q^{\deg p})^{-j}x^i)^{-1}}}_{x\mapsto x^{\deg p}} \\
&=\prod_{i,j\geq 1}(1-(q^{-j}x^i)^{\deg p})^{-1}.
\end{align}

It follows that
\begin{align}
\Zhat_X(x) &= \prod_{p\in \cl(X)}\Zhat_{\Ohat_{X,p}}(x)\\
&= \prod_{i,j\geq 1} \prod_{p\in \cl(X)} (1-(q^{-j}x^i)^{\deg p})^{-1}\\
&= \prod_{i,j\geq 1} Z_X(q^{-j}x^i).
\end{align}
\end{enumerate}
\end{proof}

\section{Properties of $H_q(x)$ and Proof of Theorem \ref{thm:A}}\label{sec:Hq}
Theorem \ref{thm:A}(b) almost follow immediately from Theorem \ref{thm:B}: let $X=\Spec \Fq[u,v]/(uv)$ be the union of $x$- and $y$-axes on a plane, and let $p$ be the origin, then $\Zhat_{R_q}(x)$ in Theorem \ref{thm:A} is $\Zhat_{X,p}(x)$. By \eqref{eq:motivic}, we have

\begin{align}
\Zhat_{X,p}(x) &=\frac{\Zhat_X(x)}{\Zhat_{X\setminus p}(x)} \\
&= \frac{(x;q^{-1})_\infty^{-2} H_q(x)}{(\Zhat_{\A^1\setminus 0}(x))^2}\\
&= \frac{(x;q^{-1})_\infty^{-2} H_q(x)}{(1-x)^{-2}}\\
&= (xq^{-1};q^{-1})_\infty^{-2} H_q(x),\label{eq:thmAb_proof}
\end{align}
where $\A^1\setminus 0$ denotes an affine line minus one point and we note that $X\setminus p$ is two copies of $\A^1\setminus 0$. The Cohen--Lenstra series of $\A^1\setminus 0$ is computed in Corollary \ref{cor:A1star}. This finishes the proof of Theorem \ref{thm:A}(b) except the claim that $q^i, i\geq 1$ are actually double poles of $\Zhat_{\Fq\bbracks{u,v}/(uv)}(x)$. This requires that $H_q(q^i)\neq 0$ for $i\geq 1$, which turns out to be elementary from the expression of $H_q(x)$; see Proposition \ref{prop:Hq_properties}(b) below.  

The proof of Theorem \ref{thm:A} is complete given the observations about $H_q(x)$ in Proposition \ref{prop:Hq_properties}. Let $t=q^{-1}$ and we define
\begin{align}
H(x;t)&:=H_q(x)=\sum_{k=0}^\infty t^{k^2}x^{2k} \frac{1}{(t;t)_k} (1-t^{k+1}x)(1-t^{k+2}x)\dots \label{eq:Hq_def} \\
& = (tx;t)_\infty \sum_{k=0}^\infty t^{k^2}x^{2k} \frac{1}{(t;t)_k (tx;t)_k.}\;\textup{(if $x\neq t^{-1},t^{-2},\dots$)}.\label{eq:Hq_def2}
\end{align}

We note that the infinite sum \eqref{eq:Hq_def} defines $H(x;t)$ in two possible ways. First, the infinite sum converges formally to a power series in $x$ and $t$ (due to the $x^{2k}$ factor). Second, if $0<t<1$ is fixed, then each summand of \eqref{eq:Hq_def} is an entire function in $x$, so \eqref{eq:Hq_def} is an infinite sum of functions. We will show that the ``formal'' sum and the ``analytic'' sum are the same.

\begin{proposition}\label{prop:Hq_properties}

For any fixed real number $0<t<1$, we have

\begin{enumerate}
\item The infinite sum \eqref{eq:Hq_def} of entire functions in $x$ converges uniformly on any bounded disc to an entire function whose Maclaurin series is the coefficient-wise limit of the sum \eqref{eq:Hq_def} of formal power series in $x$. 

\item $H(x;t)>0$ if $x<t^{-1}$ or $x=t^{-i}, i=1,2,\dots$.

\item $H(1;t)=1$.
\item $H(-1;t)=(-t;t)_\infty (-t;t^2)_\infty$.
\end{enumerate}
\end{proposition}

\begin{proof}~
\begin{enumerate}
\item Fix a bounded disc $\abs{x}\leq M$. Then as $k$ goes to infinity, the factor $\dfrac{1}{(t;t)_k} (1-t^{k+1}x)(1-t^{k+2}x)\dots$ has a uniform bound for $\abs{x}\leq M$ that only depends on $M$. The uniform convergence of the sum \eqref{eq:Hq_def} follows from the convergence of $\sum t^{k^2}M^k$. Therefore, the sum \eqref{eq:Hq_def} defines an entire function.

To find its Maclaurin series, consider the sequence of partial sums of \eqref{eq:Hq_def}. The assertion of (a) follows from the fact that if a sequence of holomorphic functions $f_k(x)$ converges uniformly to a holomorphic function $f(x)$ on a disc $D$ centered at $x=0$, then the Maclaurin series of $f_k(x)$ must converge to the Maclaurin series of $f(x)$ coefficient-wise.  For a proof, we recall that the $n$-th Maclaurin coefficient of $f(x)$ is given by $n! f^{(n)}(0)$. For any $n$, the sequence $f_k^{(n)}(x)$ converges uniformly to $f^{(n)}(x)$ on compact subsets of $D$ (see for instance \cite[Theorem 10.28]{rudinrealandcomplex}), so we have $n! f_k^{(n)}(0) \to n! f^{(n)}(0)$ as $k\to \infty$. 

\item If $x<t^{-1}$, then every term of \eqref{eq:Hq_def} is positive. If $x=t^{-i}, i=1,2,\dots$, then
\begin{align}
H(t^{-i};t)&=\sum_{k=0}^\infty \frac{t^{k^2}(t^{-i})^{2k}}{(1-t)\dots (1-t^k)} (1-t^{k+1}t^{-i})(1-t^{k+2}t^{-i})\dots\\
&=\sum_{k=i}^\infty \frac{t^{k^2}(t^{-i})^{2k}}{(1-t)\dots (1-t^k)} (1-t^{k+1-i})(1-t^{k+2-i})\dots
\end{align}
and every term in the last sum is positive. 

\item 
Using \eqref{eq:Hq_def2}, we have
\begin{align}
H(1;t)=(t;t)_\infty \sum_{k=0}^\infty t^{k^2} \frac{1}{(t;t)_k (t;t)_k},
\end{align}
which is equal to $1$ by the following standard identity due to Euler; see \cite[p.\ 21, (2.2.9)]{andrewspartitions}.
\begin{equation}\label{eq:euler1}
\sum_{k=0}^\infty \frac{t^{k^2}}{(t;t)_k^2} = \frac{1}{(t;t)_\infty}. 
\end{equation}

\item
To compute $H(-1;t)$, we need the following identities:
\begin{gather}
(t^2;t^2)_n=(t;t)_n(-t;t)_n\label{eq:pochhammersquare}\\
\sum_{n=0}^\infty \frac{t^{{n\choose 2}}x^n}{(t;t)_n} = (-x;t)_\infty,\label{eq:euler2}
\end{gather}
where the first one is elementary:
\begin{align}
(t^2;t^2)_n&=\prod_{i=1}^n (1-t^{2i}),\\
&=\prod_{i=1}^n (1-t^i)(1+t^i) = (t;t)_n (-t;t)_n,
\end{align}
and the second one is due to Euler; see \cite[p.\ 19, (2.2.6)]{andrewspartitions}. Now we have
\begin{align}
H(-1;t)&=(-t;t)_\infty \sum_{k=0}^\infty \frac{t^{k^2}}{(t;t)_k(-t;t)_k}  &\\
&=(-t;t)_\infty \sum_{k=0}^\infty \frac{(t^2)^{{k\choose 2}} t^k}{(t^2;t^2)_k} & \text{(by \eqref{eq:pochhammersquare})}\\
&=(-t;t)_\infty (-t;t^2)_\infty. & \text{(by \eqref{eq:euler2} with $x\mapsto t$, $t\mapsto t^2$)}
\end{align}
\end{enumerate}
\end{proof}

The proof of Theorem \ref{thm:A} now follows directly from Theorem \ref{thm:B} and Proposition \ref{prop:Hq_properties}. Recall the notation $R_q$ and $H_q(x)$ in the statement of Theorem \ref{thm:A}.

\begin{proof}[Proof of Theorem \tu{\ref{thm:A}}]
~
\begin{enumerate}
\item Since $H_q(x)$ is entire, the factorization \eqref{eq:factorize} defines the meromorphic continuation of $\Zhat_{R_q}(x)$ to all of $\C$.
\item We have proved the factorization formula \eqref{eq:factorize} in the computation leading to \eqref{eq:thmAb_proof}. We proved the entireness of $H_q(x)$ in Proposition \ref{prop:Hq_properties}(a). The fact that $x=q^i, i=1,2,\dots$ are double poles of $\Zhat_{R_q}(x)$ follows from Proposition \ref{prop:Hq_properties}(b).
\item The evaluation of $\Zhat_{R_q}(1)$ follows from Proposition \ref{prop:Hq_properties}(c). The evaluation of $\Zhat_{R_q}(-1)$ follows from Proposition \ref{prop:Hq_properties}(d) and the following computation:
\begin{align}
\Zhat_{R_q}(-1) &= (-t;t)_\infty^{-2} H(-1;t)\\
&= \frac{(-t;t^2)_\infty}{(-t;t)_\infty}\\
&= \frac{(1+t)(1+t^3)(1+t^5)\dots}{(1+t)(1+t^2)(1+t^3)\dots}\\
&=\frac{1}{(1+t^2)(1+t^4)(1+t^6)\dots}\\
&=\frac{1}{(-t^2;t^2)_\infty}.
\end{align}
\end{enumerate}
\end{proof}

\section{Further discussion on $H_q(x)$}\label{sec:further}
The factor $H_q(x)$ that appears in the factorization of $\Zhat_{X,p}(x)$ (where $(X,p)$ is a nodal singularity) is mostly mysterious. We note that if Question \ref{q:main}(b) has a positive answer, then we can attach a meaningful holomorphic function $H_{X,p}(x)$ to a curve singularity $(X,p)$. In particular, $H_q(x)$ would be $H_{X,p}(x)$ associated to a nodal singularity $(X,p)$. It is natural to ask about the properties of $H_{X,p}(x)$ for a general singularity $(X,p)$, and what they reveal about the geometry of $X$ at $p$. This section discusses possible analytic properties of $H_q(x)$, aiming at providing clues to the questions above.

The mysterious function $H_q(x)=H(x;t)$ (where $t=q^{-1}\in (0,1/2]$ is fixed) appears to share some analytical features with the \emph{partial theta function} 
$\partialtheta(x;t):=\sum_{n=0}^\infty t^{n^2}x^n$. We summarize several notable properties of the partial theta function; we refer the readers to an excellent survey paper \cite{warnaarpartial} where many references are listed. The partial theta function satisfies the functional equation 
\begin{equation}\label{eq:func_eq_theta}
\partialtheta(x;t)-tx\partialtheta(t^2 x;t)=1.
\end{equation}

An important property of the partial theta function is having \emph{smooth coefficients}. An entire function $f(x)=\sum a_nx^n$ is said to have smooth coefficients if $\lim_{n\to \infty} a_n^2/(a_{n-1} a_{n+1})$ converges. The partial theta function has smooth coefficients because $a_n^2/(a_{n-1} a_{n+1})$ is constant.

Having smooth coefficients is the reason behind many other analytic behaviors of $\partialtheta(x;t)$, such as having roots distributed in an ``almost geometric sequence'', belonging to the Laguerre--P\'{o}lya class, etc.; see \cite{warnaarpartial} for an excellent survey paper on this topic. Therefore, having smooth coefficients is a key feature to look for when comparing $H(x;t)$ to $\partialtheta(x;t)$. 

Based on numerical observations, the roots of $H(x;t)$ appear to be imaginary, and $H(x;t)$ does not appear to have smooth coefficients. However, the even-degree terms and odd-degree terms of $H(x;t)$ appear to have smooth coefficients and real roots. For any power series $f(x)=\sum a_nx^n$, denote
\begin{equation}
\ell_x f(x) := \lim_{n\to \infty} \frac{a_n^2}{a_{n-1} a_{n+1}}.
\end{equation}

Our observations suggest the following conjecture. 

\begin{conjecture}
The function $H(x;t)$ satisfies the following properties:

\begin{enumerate}
\item As a power series in $x$ and $t$, we have
\begin{equation}
H(x;t)=\sum_{n=0}^\infty (-1)^n t^{\ceil{n^2/4}} (1+O(t)) x^n.
\end{equation}
\item Let $F(x;t)$ and $G(x;t)$ be defined such that $H(x;t)=F(x^2;t)+xG(x^2;t)$. Then both $F(x;t)$ and $G(x;t)$ have smooth coefficients. Moreover, both $\ell_x F(x;t)$ and $\ell_x G(x;t)$ are equal to $t^2$.
\end{enumerate}
\end{conjecture}

\begin{question}
Does $F(x;t)$ (or $G(x;t)$) satisfy a functional equation, possibly similar to \eqref{eq:func_eq_theta}, the functional equation for $\partialtheta(x;t)$?
\end{question}

We note $\ell_x \partialtheta(x;t)=t^2$, the same as the conjectured value of $\ell_x F(x;t)$ and $\ell_x G(x;t)$. 

Apart from the similarity to the partial theta function, another motivation why we look for a functional equation for $H(x;t)$ is an observation by Cohen and Lenstra \cite[\S 7]{cohenlenstra1984heuristics}, where they find a functional equation for an entire function built from the Cohen--Lenstra zeta function of a Dedekind domain.

\subsection*{Acknowledgements}
The author thanks Jason Bell, Jim Bryan, Ken Ono and John Stembridge for fruitful conversations. The author thanks Jeffrey Lagarias, Mircea Musta\c{t}\u{a} and Michael Zieve for fruitful conversations and helpful comments on earlier drafts. The author thanks the referee for informing the author about certain related work. The author thanks the support of National Science Foundation grants DMS-1701576 and DMS-1840234.

\bibliography{paper_cohen_lenstra_zeta_v2}

\end{document}